\documentclass[leqno]{article}
\usepackage{a4wide,amsmath,amssymb,amsfonts,mathrsfs,exscale,graphics,amsthm,colonequals,comment,ifthen,url,breakurl,enumerate,etoolbox,bbold,todonotes}
\usepackage{hyperref}
\usepackage[curve,matrix,arrow,cmtip]{xy}
\newdir^{ (}{{}*!/-3pt/\dir^{(}}   % Usage: \ar@{^{ (}->}[r]^{j}
\newdir_{ (}{{}*!/-3pt/\dir_{(}}   % Usage: \ar@{_{ (}->}[r]^{j}    
\NoComputerModernTips

\newtoggle{thesis}
\togglefalse{thesis}

\numberwithin{equation}{section}

\theoremstyle{plain}
\newtheorem{theorem}[equation]{Theorem}%[section]

\newtheorem{lemma}[equation]{Lemma}

\newtheorem{proposition}[equation]{Proposition}

\theoremstyle{definition}
\newtheorem{definition}[equation]{Definition}

\newtheorem{construction}[equation]{Construction}
\newtheorem{remark}[equation]{Remark}

\DeclareMathOperator{\forg}{forg}

 \DeclareMathOperator{\End}{End}
 \DeclareMathOperator{\Spec}{Spec}

 \DeclareMathOperator{\Hom}{Hom}
 
\let\into\hookrightarrow

\newcommand{\defeq}{\colonequals}

\newcommand{\Gm}[1][\empty]{
  \ifthenelse{\equal{#1}{\empty}}
    {\mathbb{G}_m}
    {\mathbb{G}_{m,#1}}}

 \newcommand{\Gred}[1][\empty]{
  \ifthenelse{\equal{#1}{\empty}}
    {G^{\text{red}}}
    {G^{\text{red},#1}}}

 \newcommand{\Rep}[1][\empty]{
  \ifthenelse{\equal{#1}{\empty}}
    {\mathop{\text{\tt Rep}}\nolimits}
    {\mathop{\text{$#1$-{\tt Rep}}}\nolimits}}

\DeclareMathOperator{\gr}{gr}
\DeclareMathOperator{\fil}{fil}

\newcommand\toover[1]{\mathrel{\smash{\overset{#1}{\to}}}}
\newcommand\varto[1]{\mathrel{\hbox to #1pt{\rightarrowfill}}}
\renewcommand{\implies}{\Rightarrow}

\let\longto\longrightarrow

\def\isoto{\stackrel{\sim}{\longto}}
 
\newcommand{\BZ}{{\mathbb{Z}}}

\newcommand{\CA}{{\mathcal A}}
\newcommand{\CB}{{\mathcal B}}
\newcommand{\CC}{{\mathcal C}}
\newcommand{\CD}{{\mathcal D}}

\newcommand{\CG}{{\mathcal G}}

\newcommand{\CL}{{\mathcal L}}
\newcommand{\CM}{{\mathcal M}}
\newcommand{\CN}{{\mathcal N}}
\newcommand{\CO}{{\mathcal O}}

\newcommand{\CS}{{\mathcal S}}
\newcommand{\CT}{{\mathcal T}}

\def\UHom{\mathop{\underline{\rm Hom}}\nolimits}
\def\UEnd{\mathop{\underline{\rm End}}\nolimits}

\def\UAut{\mathop{\underline{\rm Aut}}\nolimits}

\def\ULie{\mathop{\underline{\rm Lie}}\nolimits}

\def\USpl{\mathop{\underline{\rm Spl}}\nolimits}

\let\phi\varphi

\DeclareMathOperator{\id}{id}

\DeclareMathOperator{\Lie}{Lie}

\def\GFF{\mathop{\tt GFF}\nolimits}
\def\Mod{\mathop{\tt Mod}\nolimits}
\def\Gr{\mathop{\tt Gr}\nolimits}
\def\Fil{\mathop{\tt Fil}\nolimits}

\DeclareMathOperator{\Lex}{Lex}
\DeclareMathOperator{\Ab}{Ab}
\DeclareMathOperator{\Set}{Set}

\DeclareMathOperator{\colim}{colim}

\newcommand{\FunEx}[1][\empty]{
\ifthenelse{\equal{#1}{\empty}}
{\operatorname{Fun}_{\operatorname{ex}}}
{\operatorname{Fun}_{\operatorname{ex},#1}}}

\newcommand{\FunLex}[1][\empty]{
\ifthenelse{\equal{#1}{\empty}}
{\operatorname{Fun}_{\operatorname{lex}}}
{\operatorname{Fun}_{\operatorname{lex},#1}}}

\newcommand{\FunRex}[1][\empty]{
\ifthenelse{\equal{#1}{\empty}}
{\operatorname{Fun}_{\operatorname{rex}}}
{\operatorname{Fun}_{\operatorname{rex},#1}}}

\newcommand{\FunColim}[1][\empty]{
\ifthenelse{\equal{#1}{\empty}}
{\operatorname{Fun}_{\operatorname{cc}}}
{\operatorname{Fun}_{\operatorname{cc},#1}}}

\newcommand{\FunLA}[1][\empty]{
\ifthenelse{\equal{#1}{\empty}}
{\operatorname{Fun}_{\operatorname{la}}}
{\operatorname{Fun}_{\operatorname{la},#1}}}

\newcommand{\op}{\text{op}}

\def\Sch{{\bf Sch}}
\begin{document}
\title{Filtered fiber functors over a general base}

%\author{Paul Ziegler}
\author{Paul Ziegler\footnote{TU Darmstadt, Darmstadt, Germany,
 {\tt ziegler@mathematik.tu-darmstadt.de}}
}

% \date{}

\maketitle
\begin{abstract}
  We prove that every filtered fiber functor on the category of dualizable representations of a smooth affine group scheme with enough dualizable representations comes from a graded fiber functor.
\end{abstract}

\section{Introduction}

This article is concerned with filtered fiber functors over general bases, extending the previous article \cite{FFF} of the author. Let $R$ be a ring, $G$ a smooth affine group scheme over $R$ and $(\Gamma,\leq)$ a totally ordered group. For any scheme $S$, we denote by $\Gr(S)$ (resp. $\Fil(S)$) the category of $\Gamma$-graded (resp. $\Gamma$-filtered) quasicoherent $\CO_S$-modules. If $\Rep^\circ G$ denotes the tensor category of dualizable representations of $G$ over $R$, then a filtered fiber functor over an $R$-scheme $S$ is an exact $R$-linear tensor functor
\begin{equation*}
  \Rep^\circ G \to \Fil(S).
\end{equation*}

One way of obtaining such a functor is to take an $R$-linear exact tensor functor $\Rep^\circ G \to \Gr(S)$ and compose it with the natural functor $\Gr(S) \to \Fil(S)$ turning a graded $\CO_S$-modules into a filtered one. Such filtered fiber functors are called splittable. Then our main theorem is:
\begin{theorem}[c.f. Theorem \ref{MainThm}]
  If $G$ has enough dualizable representations, then every filtered fiber functor on $\Rep^\circ G$ over an affine $R$-scheme $S$ is splittable.
\end{theorem}

A different proof of a generalization of this result was recently given by Wedhorn in \cite{wedhorn2023extension}, c.f. in particular \cite[Subsection 3.3]{wedhorn2023extension}. For many reductive group schemes, this result was previously proved by Cornut in \cite{CornutFiltrations}. In case the base ring $R$ is a field, this was proved in for many group schemes in \cite[Subsection IV.2.4]{SaavedraRivano} by Deligne and for all group schemes by the author in \cite{FFF}.

As with our previous work \cite{FFF}, this result is motivated by applications for example in the theory of Shimura varieties.

 Our proof proceeds by reduction to the case where $R$ is a field. To do this, we need to construct the base change of a filtered fiber functor with respect to a change $R \to R'$ of the base ring over which we consider the category $\Rep^\circ G$. For this, we first collect some facts on exact categories, the Gabriel-Quillen embedding of such categories into an abelian category, and the base change of cocomplete linear abelian categories. This is then applied to the exact category $\Rep^\circ G$.
\paragraph{Acknowledgment}
The author thanks Johannes Ansch\"utz for conversations regarding the content of this article and an anonymous referee for helpful comments. The author acknowledges support by the ERC in form of Eva Viehmann's Consolidator Grant 770936: Newton Strata, by the Deutsche Forschungsgemeinschaft (DFG) through the Collaborative Research Centre TRR 326 "Geometry and Arithmetic of Uniformized Structures", project number 444845124 and by the LOEWE professorship in Algebra, project number LOEWE/4b//519/05/01.002(0004)/87.
\section{Preliminaries}

\subsection{Exact Categories}

By an exact category we mean such a category in the sense of Quillen, c.f. \cite[Def. 2.1]{Buhler}.

\begin{lemma}[{\cite[Lemma 10.20]{Buhler}}] \label{ExtClosed}
  Let $\CA$ be an abelian category and $\CB \subset \CA$ a full additive subcategory whose class of objects is closed under extensions in $\CA$. The class of sequences in $\CB$ which are short exact in $\CA$ makes $\CB$ into an exact category.
\end{lemma}
We fix a small exact category $\CA$.
\begin{definition}
  Let $\CC$ be an additive category admitting kernels. An additive functor $F\colon \CA \to \CC$ is \emph{left exact} if for every exact sequence
  \begin{equation*}
    0 \to X \to Y \to Z \to 0
  \end{equation*}
in $\CA$ the morphism $F(X) \to F(Y)$ is a kernel of $F(Y) \to F(Z)$ in $\CC$.

 We denote by $\FunLex(\CA,\CC)$ the category whose objects are given by such functors and whose morphisms are given by arbitrary natural transformations of such functors.

Dually for an additive category $\CC$ admitting cokernels we define right exact functors $\CA \to \CC$ as left exact functors $\CA^{\text{op}} \to \CC^{\text{op}}$ and let $\FunRex(\CA,\CC)$ be the category of such functors.
\end{definition}

\begin{definition}
  We denote by $\Lex(\CA)$ the category $\FunLex(\CA^\op,\Ab)$ of left exact functors from $\CA^\op$ to the category $\Ab$ of abelian groups. 
\end{definition}

\begin{theorem}[{c.f. \cite[Appendix A]{Buhler}}] \label{LexProperties}
  \begin{enumerate}[(i)]
  \item  The category $\Lex(\CA)$ is a Grothendieck category. In particular it is abelian, complete and cocomplete. 
  \item The Yoneda functor $\CA \to \Lex(\CA),\; X \mapsto \Hom(\underline{\;\;},X)$ is exact and fully faithful and reflects exactness.

  \end{enumerate}
\end{theorem}
We will always consider $\CA$ as a full subcategory of $\Lex(\CA)$ via the Yoneda embedding.

\begin{definition}
  \begin{enumerate}[(i)]
  \item Let $\CA$ be an abelian category. A set $\CS$ of objects of $\CA$ \emph{generates} $\CA$ if for every pair of parallel morphisms $f,g \colon Y \to Z$ in $\CA$ the following condition is satisfied: If for every morphism $h\colon X \to Y$ with $X \in \CS$ the compositions $f\circ h$ and $g\circ h$ are equal, then $f$ and $g$ are equal.
  \item Let $\CC$ be a category which admits all filtered colimits. An object $X \in \CC$ is \emph{finitely presented} if the functor $\Hom(X,\underline{\;\;})\colon \CC \to \Set$ preserves filtered colimits.
  \end{enumerate}
\end{definition}

\begin{proposition}[{\cite[2.17]{MR3303245}}] \label{FpProps}
 The objects of $\CA$ are finitely presented in $\Lex(\CA)$ and generate $\Lex(\CA)$.
\end{proposition}

\begin{definition}
  For $F \in \Lex(\CA)$, we define the category $\CA \downarrow F$ as follows:
  \begin{enumerate}[(i)]
  \item Objects consist of an object $X \in \CA$ together with a morphism $X \to F$ in $\Lex(\CA)$.
  \item Morphisms are given by commutative diagrams
    \begin{equation*}
      \xymatrix{
        X \ar[rd] \ar[rr] & & X' \ar[ld] \\
        & F &
}
    \end{equation*}
in $\Lex(\CA)$.
  \end{enumerate}
\end{definition}

The following follows by a standard argument from the fact that $\CA \to \Lex(\CA)$ is the Yoneda embedding:
\begin{lemma} \label{colims}
  \begin{enumerate}[(i)]
  \item For $F \in \Lex(\CA)$, the natural morphism
    \begin{equation*}
       \underset{X \in \CA \downarrow F}{\colim} \;X\to F,
    \end{equation*}
where the colimit is formed in $\Lex(\CA)$, is an isomorphism.
\item For objects $F,G \in \Lex(\CA)$, the natural homomorphism
  \begin{equation*}
    \underset{X\in \CA \downarrow F}{\lim} \; \underset{Y \in \CA \downarrow G}{\colim} \; \Hom(X,Y) \to \Hom(F,G)
  \end{equation*}
is an isomorphism.
  \end{enumerate}
\end{lemma}

\begin{definition}
  For cocomplete categories $\CC$ and $\CD$, we denote by $\FunColim(\CC,\CD)$ the category of colimit-preserving functors $\CC \to \CD$ and by $\FunLA(\CC,\CD) \subset \FunColim(\CC,\CD)$ the full subcategory of functors which admit a right adjoint.
\end{definition}
The embedding $\CA \to \Lex(\CA)$ has the following universal property:
\begin{proposition} \label{LexUP}
  Let $\CC$ a cocomplete additive category. The functors 
  \begin{equation*}
    \FunLA(\Lex(\CA),\CC) \to \FunColim(\Lex(\CA),\CC) \to \FunRex(\CA,\CC),
  \end{equation*}
the second of which is induced by the exact functor $\CA \to \Lex(\CA)$, are equivalences of categories.

\end{proposition}
\begin{proof}
  Consider a right-exact functor $F\colon \CA \to \CC$. We extend $F$ to a functor $F^*\colon \Lex(\CA) \to \CC$ by 
  \begin{equation*}
    F^*(X) \defeq \underset{Y \in \CA \downarrow X}{\colim} F(Y),
  \end{equation*}
for $X \in \Lex(\CA)$ and similarly on the level of morphisms using Lemma \ref{colims} (ii). We construct a right adoint $F_*\colon \CC \to \Lex(\CA)$ to $F^*$ as follows: For $C \in \CC$, let $F_*(C)$ be the functor $\CA^\op \to \Ab,\; X \mapsto \Hom_\CC(F(X),C)$. The fact that $F$ is right-exact implies that this functor is left exact. Thus it defines an object $F_*C \in \Lex(\CA)$. This construction is naturally functorial in $C$.

 Hence we have proved essential surjectivity of the functors in question. Full faithfulness follows from Lemma \ref{colims}.

\end{proof}

\begin{proposition} \label{LexChar}
  Let $\CA$ be a Grothendieck abelian category and $\CB$ a full subcategory of $\CA$ whose objects are finitely presented in $\CA$, generate $\CA$ and are closed under extensions in $\CA$. Endow $\CB$ with the unique exact structure for which the inclusion $\CB \to \CA$ is exact and reflects exactness given by Lemma \ref{ExtClosed}. The inclusion $\CB \into \CA$ extends to an equivalence $\Lex(\CB) \cong \CA$ which is unique up to a canonical isomorphism.
\end{proposition}
\begin{proof}
  This is a special case of \cite[Proposition 2.3]{MR3303245}, using the fact that in our situation the category $\operatorname{Sh}(\CC,\CT^\text{add}_\Lambda)$ appearing there is canonically isomorphic to $\Lex(\CB)$ by Subsections 2.4 and 2.5 of \cite{MR3303245}.
\end{proof}

\subsection{Tensor Categories}
We refer to symmetric monoidal categories (resp. symmetric monoidal functors, resp. monoidal natural transformations) as tensor categories (resp. tensor functors, resp. tensor morphisms).

\begin{definition}
  Let $R$ be a ring.
  \begin{enumerate}[(i)]
  \item  An \emph{$R$-linear tensor category} is a tensor category $\CT$ which is additionally equiped with a $R$-linear structure for which the tensor product functor $\otimes\colon \CT \times \CT \to \CT$ is $R$-bilinear. 
  \item  An \emph{exact tensor category} is a tensor category which is additionally equipped with an exact structure for which the tensor product functor $\otimes\colon \CT \times \CT \to \CT$ is exact in each variable.
      \item  A \emph{cocomplete tensor category} is a tensor category which is cocomplete and for which the tensor product functor $\otimes\colon \CT \times \CT \to \CT$ is cocontinuous in each variable.

  \item Let $f\colon R\to S$ be a ring homomorphism. For an $R$-linear tensor category $\CT$ and an $S$-linear tensor category $\CT'$, a \emph{$f$-linear tensor functor} $\CT\to \CT'$ is a functor which is both monoidal and $f$-linear. In case $f=\id_R$, we will call such functors simply $R$-linear tensor functors.
  \end{enumerate}
\end{definition}

\begin{construction} \label{LexTensor}
  Let $\CT$ be an exact tensor category. Then $\Lex(\CT)$ is in a natural way a cocomplete tensor category:

  By Proposition \ref{LexUP} the tensor functor $\CT \times \CT \to \CT$ induces a cocontinous functor $\Lex(\CT) \times \Lex(\CT) \cong \Lex(\CT \times \CT) \to \Lex(\CT)$. The various constraints making up the tensor category $\CT$ can be described as natural transformations between certain functors. Hence again by Proposition \ref{LexUP} these extend uniquely to $\Lex(\CT)$ and make this category into an abelian tensor category.

Lemma \ref{colims} together with \cite[A.22]{Buhler} imply that every epimorphism in $\Lex(\CT)$ is a colimit of strict epimorphisms in $\CT$. This implies that the tensor product functor on $\Lex(\CT)$ is again cocontinuous in each variable.
\end{construction}

\subsection{Base change}
In this subsection we consider a cocomplete $R$-linear category $\CA$ and its base change along a ring homomorphism $R \to R'$. We recall the construction of this base change from \cite[II.1.5.2]{SaavedraRivano}:
\begin{construction} \label{BaseChange}
  The category $\CA_{R'}$ is defined as follows: Its objects are pairs $(X,a)$ consisting of an object $X \in \CA$ and a homomorphism $a\colon R' \to \End(X)$ of $R$-algebras. Morphisms $(X,a) \to (X',a)$ are morphisms $X \to X'$ in $\CA$ compatible with $a$ and $a'$ in the natural way.

There is a natural functor $\CA \to \CA_{R'},\; X \mapsto X \otimes_R R'$ which is left adjoint to the forgetful functor $\CA_{R'} \to \CA$.
\end{construction}

\begin{proposition}[{\cite[II.1.5.3.1]{SaavedraRivano}}] \label{CCBCUP}
  The functor $\CA \to \CA_{R'}$ from Construction \ref{BaseChange} satisfies the following $2$-universal property: For every $R'$-linear cocomplete category $\CB$, the functor $\CA \to \CA_{R'}$ induces an equivalence 
  \begin{equation*}
    \FunColim[R'](\CA_{R'},\CB) \to \FunColim[R](\CA,\CB)
  \end{equation*}
from the category of cocontinuous $R'$-linear functors $\CA_{R'} \to \CB$ to the category of cocontinuous $R$-linear functors $\CA \to \CB$.
\end{proposition}
Here, given a cocontinuous $R$-linear functor $F\colon \CA \to \CB$, the cocontinuous $R'$-linear extension $F'\colon \CA_{R'} \to \CB$ sends a pair $(X,a)$ to the largest quotient of $F(X)$ on which the two actions of $R'$ coincide.

Now assume in addition that $\CA$ is a cocomplete $R$-linear tensor category. Via the universal property from Proposition \ref{CCBCUP}, the tensor product functor $\CA \times \CA \to \CA$ induces an $R'$-linear functor $\CA_{R'} \times \CA_{R'}=(\CA\times\CA)_{R'} \to \CA_{R'}$. By \cite[II.1.5.4]{SaavedraRivano}, this functor makes $\CA_{R'}$ into a cocomplete $R'$-linear tensor category. Then the functor $\CA \to \CA_{R'}$ is a tensor functor which has the following $2$-universal property:

\begin{proposition} \label{TBCUP}
  For every $R'$-linear cocomplete tensor category $\CB$, the functor $\CA \to \CA_{R'}$ induces an equivalence 
  \begin{equation*}
    \FunColim[R']^\otimes(\CA_{R'},\CB) \to \FunColim[R]^\otimes(\CA,\CB)
  \end{equation*}
from the category of cocontinuous $R'$-linear tensor functors $\CA_{R'} \to \CB$ to the category of cocontinuous $R$-linear tensor functors $\CA \to \CB$.
\end{proposition}

\subsection{Categories of representations of group schemes}
For a ring $R$ (resp. a scheme $S$), we denote by $\Mod(R)$ (resp. $\Mod(S)$) the category of $R$-modules (resp. of quasi-coherent $\CO_S$-modules).

Let $R$ be a ring and $G$ an affine group scheme over $\Spec(R)$. 
\begin{definition}
  \begin{enumerate}[(i)]
  \item We denote by $\Rep G$ the category of representations of $G$ on arbitrary $R$-modules.
  \item We denote by $\Rep^\circ G \subset \Rep G$ the full subcategory of dualizable objects, that is of representations of $G$ on finitely generated projective $R$-modules.
  \item We denote the forgetful functor to $\Mod(R)$ on any of these categories by $\omega_G$.
  \item For a ring homomorphism $f\colon R \to S$, an affine group scheme $H$ over $\Spec(S)$ and a homomorphism $h\colon H \to  G_S$, we denote by $h^*$ the induced $f$-linear tensor functors $\Rep G \to \Rep H$ and $\Rep^\circ G \to \Rep^\circ H$.
  \end{enumerate}
\end{definition}
The subcategory $\Rep^\circ G \subset \Rep G$ is closed under tensor products. On $\Rep^\circ G$ there exists an exact structure given by those sequences for which the underlying sequence of $R$-modules is exact. We will always endow $\Rep^\circ G$ with this exact structure. Equivalently, this exact structure is induced via Lemma \ref{ExtClosed} from the embedding $\Rep^\circ G \into \Rep G$. 

\begin{lemma} \label{EDConds}
  The following are equivalent:
  \begin{enumerate}[(i)]
  \item Every object of $\Rep G$ admits an epimorphism from a direct sum of objects of $\Rep^\circ G$.
  \item The inclusion $\Rep^\circ G \into \Rep G$ induces an equivalence $\Lex(\Rep^\circ G)\cong \Rep G$.
  \end{enumerate}
\end{lemma}
\begin{proof}
  The implication $(i) \implies (ii)$ follows from Lemma \ref{colims}.

  The fact that the objects of $\Rep^\circ G$ are dualizable implies that they are finitely presented in $\Rep G$. Hence the implication $(ii) \implies (i)$ is a special case of Proposition \ref{LexChar}.
\end{proof}

\begin{definition}
  We say that $G$ has enough dualizable representations if the equivalent conditions of Lemma \ref{EDConds} are satisfied.
\end{definition}

\begin{remark}
  \begin{enumerate}[(i)]
  \item  Assume that $S$ is a regular separated noetherian scheme. Then $G$ always has enough dualizable representation if $S$ is of dimension $\leq 1$ (c.f. \cite[2.4]{MR0893468}) or if $S$ is of dimension $\leq 2$ and $G$ is smooth with connected fibers (c.f. \cite[2.5]{MR0893468}).
  \item In case $G$ is reductive, a criterion for $G$ to have enough dualizable representations is given in \cite{MR4460258}.
  \end{enumerate}

\end{remark}
For an $R$-scheme $S$, a fiber functor on $\Rep^\circ G$ over $S$ is an exact $R$-linear tensor functor $\omega\colon \Rep^\circ G \to \Mod(S)$. %We denote by $\omega_G \colon \Rep^\circ G \to \Mod_R$ the standard fiber functor (i.e. the forgetful functor).

\begin{definition} \label{TEnd}
  Let $R' \to R''$ be an homomorphism of $R$-algebra and let $\omega\colon \Rep^\circ G \to \Mod(R')$ be a fiber functor.
  \begin{enumerate}[(i)]
  \item Let $\UEnd(\omega)(R'')$ be the $R''$-module of $R''$-linear natural transformations $\omega_{R''} \to \omega_{R''}$.
  \item Let $\UEnd^\otimes(\omega)(R'')$ be the submodule of $\UEnd(\omega)(R'')$ consisting of those natural transformations $h$ such that for all $X$ and $Y$ in $\Rep^\circ \CG$, the associated endomorphism $h_{X\otimes Y}$ of $\omega(X\otimes Y)_{R''}$ is equal to $h_{X} \otimes \id_{\omega(Y)_{R''}} + \id_{\omega(X)_{R''}} \otimes h_{Y}$.
  \item For varying $R''$, these modules form natural sheaves $\UEnd^\otimes(\omega) \subset \UEnd(\omega)$ over $\Spec(R')$.
  \item We endow $\UEnd^\otimes(\omega)$ with the natural conjugation action of $\UAut^\otimes(\omega)$.
  \end{enumerate}
\end{definition}

\begin{proposition}[{\cite[Thm. 3.2]{VFF}}] \label{TannakaLie}
  There exists a natural isomorphism between the Lie algebra sheaf $\ULie(\UAut^\otimes(\omega))$  and the sheaf $\UEnd^\otimes(\omega)$. This isomorphism is equivariant with respect to the adjoint action of $\UAut^\otimes(\omega)$ on $\ULie(\UAut^\otimes(\omega))$ and the above $\UAut^\otimes(\omega)$-action on $\UEnd^\otimes(\omega)$.
\end{proposition}

The following result generalizes Deligne's theorem on the fqpc-local isomorphy of any two fiber functors on a Tannakian category. 
\begin{theorem}[{Lurie (\cite[5.11]{LurieTannaka}}]  \label{FpqcIso}
  If $G$ is smooth and has enough dualizable representations, then any two fiber functors $\Rep^\circ G \to \Mod(S)$ are isomorphic fqpc-locally on $S$.
\end{theorem}

\section{Graded Fiber Functors}

Let $R$ be a ring, let $S$ be a scheme over $\Spec(R)$ and let $G$ be an affine group scheme over $\Spec(R)$ which has enough dualizable representations.

We fix a totally ordered abelian group $(\Gamma,\leq)$ which we write additively. We let $D^\Gamma$ be multiplicative group scheme over $\Spec(\BZ)$ with character group $\Gamma$.

\begin{definition}
  \begin{enumerate}[(i)]
%  \item We denote by $\Mod(S)$ the category of quasi-coherent $\CO_S$-modules.
  \item  We denote by $\Gr(S) \cong \Rep D^\Gamma_S$ the category of $\Gamma$-graded quasi-coherent $\CO_S$-sheaves with its standard structure as an abelian tensor category.
  \item We denote by $\forg\colon \Gr(S) \to \Mod(S)$ the forgetful functor.
  \end{enumerate}
\end{definition}

\begin{definition}
  \begin{enumerate}[(i)]
  \item A \emph{graded fiber functor} $\gamma$ on $\Rep^\circ G$ over $S$ is an exact $R$-linear tensor functor $\Rep^\circ G \to \Gr(S)$.
  \item A \emph{morphism} between two graded fiber functors is a tensor morphism.
  \item We denote the resulting \emph{category of graded fiber functors} on $\Rep^\circ G$ over $S$ by $\UHom^\otimes(\Rep^\circ G,\Gr)(S)$.
  \item For a morphism of schemes $S' \to S$, composition with the pullback functor $\Gr(S) \to \Gr(S')$ defines a pullback functor $\UHom^\otimes(\Rep^\circ G,\Gr)(S) \to \UHom^\otimes(\Rep^\circ G,\Gr)(S')$. With these pullback functors the categories $\UHom^\otimes(\Rep^\circ G,\Gr)(S)$ for varying $S$ form a fibered category over $\Spec(R)$ which we denote by $\UHom^\otimes(\Rep^\circ G,\Gr)$.
  \end{enumerate}
\end{definition}

\begin{definition}
  For any graded fiber functor $\gamma$ on $\Rep^\circ G$ over $S$, we let $\UAut^\otimes(\gamma)$ be the functor $(\Sch/S) \to (\text{Groups})$ which sends $S'\to S$ to the group of tensor automorphisms of $(\gamma)_S$ and morphisms to pullbacks. 
\end{definition}

\begin{proposition} \label{GammaAutRepr}
  For any graded fiber functor $\gamma$ on $\Rep^\circ G$ over $S$, the functor $\UAut^\otimes(\gamma)$ is representable by a scheme which is affine over $S$.
\end{proposition}
\begin{proof}
  The functor $\forg\colon \Gr(S) \to \Mod(S)$ induces a monomorphism $\UAut^\otimes(\gamma) \to \UAut^\otimes(\forg \circ \gamma)$ to the functor $\UAut^\otimes(\forg\circ\gamma)$ which is representable by a scheme which is affine over $S$. (This representability follows e.g. from Theorem \ref{FpqcIso} which implies that $\UAut^\otimes(\forg\circ\gamma)$ is isomorphic to $G$fpqc-locally on $S$ and hence representable by flat descent.) So it suffices to show that this monomorphism is representable by a closed immersion. This can be shown by the argument used in the proof of \cite[Theorem 4.6]{FFF}.
\end{proof}

\begin{construction}
  Since $\Gr(S)$ can be identified with the category of representations of $D^\Gamma$ over $S$, by \cite[II.3.1.1]{SaavedraRivano} the action of $D^\Gamma$ gives an isomorphism
  \begin{equation*}
    D^\Gamma_S \isoto \UAut^\otimes(\forg\colon \Gr(S) \to \Mod(S)).
  \end{equation*}
Thus to any fiber functor $\gamma$ on $\Rep^\circ G$ over $S$ we can associate the cocharacter
\begin{equation*}
  \chi(\gamma)\colon D^\Gamma_S \cong \UAut^\otimes(\forg\colon \Gr(S) \to \Mod(S)) \to \UAut^\otimes(\forg\circ \gamma).
\end{equation*}
\end{construction}

\begin{definition}
  Let $\GFF(S)$ be the following category: 

Objects are pairs $(\omega,\chi)$ consisting of a fiber functor $\omega$ on $\Rep^\circ G$ over $S$ and a homomorphism $\chi\colon D^\Gamma_S \to \UAut^\otimes(\forg\circ \gamma)$ over $S$.

A morphism $(\omega,\chi) \to (\omega',\chi')$ is a tensor morphism $\lambda\colon \omega \to \omega'$ for which the following diagram commutes. Here the vertical morphism sends a tensor automorphism $\psi$ of $\omega$ to the tensor automorphism $\lambda\circ\psi\circ\lambda^{-1}$ of $\omega'$, which is well-defined since $\lambda$ is an isomorphism by \cite[I.5.2.3]{SaavedraRivano}.

\begin{equation*}
  \xymatrix@R=1pt{
    & \UAut^\otimes(\omega) \ar[dd] \\
    D^\Gamma_S \ar[ur]^{\chi} \ar[dr]_{\chi'} & \\
    & \UAut^\otimes(\omega') \\
  }
\end{equation*}
\end{definition}

The following generalizes a result of Saavedra Rivano for Tannakian categories, c.f. \cite[IV.1.3]{SaavedraRivano}.
\begin{theorem}
  Every graded fiber functor $\gamma$ on $\Rep^\circ G$ over $S$ fits into the following commutative diagram:
  \begin{equation*}
    \xymatrix{
      \Rep^\circ G \ar[rd]^\gamma \ar[d] & \\
      \Rep^\circ \UAut^\otimes(\forg\circ\gamma) \ar[r]^{\chi(\gamma)} & \Gr(S)
    }
  \end{equation*}

Hence the natural functor $\UHom^\otimes(\Rep^\circ G,\Gr)(S) \to \GFF(S)$ which sends $\gamma$ to $(\forg\circ\gamma,\chi(\gamma))$ is an equivalence of categories.
\end{theorem}
\begin{proof}
  The commutativity of the diagram follows from the construction of $\chi(\gamma)$. Then using this diagram one readily defines a functor $\GFF(S) \to \UHom^\otimes(\Rep^\circ G,\Gr)(S)$ which is inverse to the functor in question.
\end{proof}
\section{Filtered Fiber Functors} \label{Filtered}
We continue with the setup from the previous section.

\subsection{Filtered Modules}
Let $S$ be a scheme.
\begin{definition}
  We denote by $\Fil(S)$ the category of $\Gamma$-filtered quasi-coherent sheaves on $S$, that is the following category:
  \begin{enumerate}[(i)]
  \item Objects of $\Fil(S)$ consist of a quasi-coherent $\CO_S$-module $\CM$ of together with an increasing filtration $(F^{\leq\gamma} \CM)_{\gamma\in \Gamma}$ by quasi-coherent submodules $F^{\leq\gamma} \CM \subset \CM$ such that $\cap_{\gamma\in\Gamma} F^{\leq\gamma} \CM=0$ and $\cup_{\gamma\in \Gamma}F^{\leq \gamma} \CM=\CM$. 
  \item Morphisms $(\CM,(F^{\leq\gamma} \CM)_\gamma) \to (\CN,(F^{\leq\gamma} \CN)_\gamma)$ are $\CO_S$-linear morphisms $h\colon \CM \to \CN$ which satisfy $h(F^{\leq\gamma}\CM) \subset F^{\leq\gamma}\CN$ for all $\gamma \in \Gamma$.
  \end{enumerate}
\end{definition}
By abuse of notation we shall often write $\CM$ for an object $(\CM,(F^{\leq\gamma} \CM)_{\gamma\in \Gamma})$ of $\Fil(S)$.

This category is naturally a tensor category: The tensor product of $\CM$ and $\CN$ is given by the module $\CM \otimes \CN$ with the filtration given by
\begin{equation*}
  F^{\leq\gamma}(\CM\otimes \CN)=\sum_{\gamma_1+\gamma_2=\gamma}\operatorname{im}(F^{\leq\gamma_1}\CM \otimes F^{\leq\gamma_2}\CN \to \CM \otimes \CN).
\end{equation*}

As in \cite[Lemma 4.2]{FFF}, the dualizable objects of $\Fil(S)$ are those for which $\CM$ is a locally free $\CO_S$-module of finite rank and for which all $F^{\leq\gamma} \CM$ are locally direct summands of $\CM$.

We call a sequence $$0 \to \CL \to \CM \to \CN \to 0$$ in $\Fil(S)$ exact if the sequence $$0 \to F^{\leq\gamma} \CL \to F^{\leq\gamma} \CM \to F^{\leq\gamma} \CN \to 0$$ is an exact sequence of quasi-coherent sheaves for all $\gamma \in \Gamma$. This makes $\Fil(S)$ into an exact category.

In case $S=\Spec(R)$ for a ring $R$ we also write $\Fil(R)$ for $\Fil(S)$. This is then a $R$-linear exact tensor category.

 For $\CM \in \Fil(S)$ and $\gamma \in \Gamma$ we let $F^{<\gamma}\CM \defeq \cup_{\gamma' < \gamma}F^{\leq\gamma'}$. There is a natural exact tensor functor $\gr\colon \Fil(S) \to \Gr(S)$ which sends a filtered module $\CM$ to its associated graded $\oplus_{\gamma\in \Gamma}F^{\leq\gamma}\CM/F^{<\gamma}\CM$.

Similarly, there is a natural exact tensor functor $\fil \colon \Gr(S) \to \Fil(S)$ which sends a graded module $\CM=\oplus_{\gamma\in \Gamma}\CM^\gamma$ to the same module $\CM$ equipped with the filtration $F^{\leq\gamma}\CM =\oplus_{\gamma'\leq \gamma}\CM^{\gamma'}$.

\subsection{Filtered Fiber Functors}

\begin{definition}
  Let $S$ be a scheme over $R$.
  \begin{enumerate}[(i)]
  \item A \emph{filtered fiber functor} on $\Rep^\circ G$ over $S$ is an exact $R$-linear tensor functor $\phi\colon \Rep^\circ G \to \Fil(S)$.
  \item An \emph{isomorphism} $\phi \to \phi'$ of such filtered fiber functors is an isomorphism of tensor functors.
  \item For a filtered fiber functor $\phi$, we call the functor $\forg\circ \phi\colon \Rep^\circ G \to \Fil(S) \to \Mod(S)$ the \emph{underlying fiber functor of $\phi$}.
  \item For a morphism of schemes $S' \to S$ and a filtered fiber functor $\phi$ on $\Rep^\circ G$ over $S$, we let $\phi_{S'}$ be the composition $\Rep^\circ G \to \Fil(S) \to \Fil(S')$.
  \end{enumerate}
\end{definition}

For a filtered fiber functor $\phi$ with underlying fiber functor $\omega$ and an object $X \in \Rep^\circ G$, we will denote the filtration of $\omega(X)$ given by $\phi$ by $(\omega(X)^{\phi\leq\gamma})_{\gamma \in\Gamma}$.
\begin{definition}
    For a filtered fiber functor $\phi$ on $\Rep^\circ G$ over $S$, we let $\UAut^\otimes(\phi)$ be the functor $(\Sch/S) \to (\text{Groups})$ which sends $S' \to S$ to the group of tensor automorphisms of $\phi_{S'}$ and morphisms to pullbacks.
\end{definition}

The following fact is proved in the same way as Proposition \ref{GammaAutRepr}.
\begin{proposition} \label{PhiAutRepr}
  Let $\phi$ be a filtered fiber functor on $\Rep^\circ G$ over $S$. The functor $\UAut^\otimes(\phi)$ is representable by a closed subgroup scheme of $\UAut^\otimes(\forg\circ\phi)$. In particular it is affine over $S$.
\end{proposition}

\begin{definition}
  Using Propositions \ref{GammaAutRepr} and \ref{PhiAutRepr}, we can associate to any filtered fiber functor $\phi$ on $\Rep^\circ G$ over $S$ the following group schemes which are affine over $S$:
  \begin{enumerate}[(i)]
  \item $G(\phi) \defeq \UAut^\otimes(\forg \circ \phi)$
  \item $P(\phi)\defeq \UAut^\otimes(\phi) \subset G(\phi)$
  \item $L(\phi)\defeq \UAut^\otimes(\gr\circ\phi)$
  \item $U(\phi)\defeq \ker(P(\phi)\toover{\gr} L(\phi))$
  \end{enumerate}
\end{definition}

\begin{lemma} \label{PhiExt}
  Let $S$ be a scheme over $R$. Every filtered fiber functor $\phi\colon \Rep^\circ G \to \Fil(S)$ extends to a cocontinuous tensor functor $\tilde\phi\colon \Rep G \to \Fil(S)$. Such an extension $\tilde\phi$ is unique up to a unique isomorphism.
\end{lemma}
\begin{proof}
  By our assumption on $G$, the inclusion $\Rep^\circ G \into \Rep G$ induces an equivalence $\Lex(\Rep^\circ G) \isoto \Rep G$. Under this equivalence, the existence and uniqueness of $\tilde \phi$ follows from the universal property given by Proposition \ref{LexUP}.
\end{proof}
\begin{proposition} \label{PhiBC}
  We consider ring homomorphisms $R \to R' \to R''$ and a filtered fiber functor $\phi\colon \Rep^\circ G \to \Fil(R'')$.

  \begin{enumerate}[(i)]
  \item   There exists a filtered fiber functor $\phi^{R'}$ on $\Rep^\circ G_{R'}$ over $R''$ which makes the following diagram commute:
  \begin{equation*}
    \xymatrix{
    \Rep^\circ G \ar[d] \ar[rd]^\phi & \\
    \Rep^\circ G_{R'} \ar[r]_{\phi^{R'}} & \Fil(R'')
    }
  \end{equation*}
  Such a functor $\phi^{R'}$ is unique up to a unique isomorphism.
  \item   The base change functor $\Rep^\circ G \to \Rep^\circ G_{R'}$ induces an isomorphism from the group of tensor automorphisms $\phi^{R'}$ to the group of tensor automorphisms of $\phi$.
  \item   The base change functor $\Rep^\circ G \to \Rep^\circ G_{R'}$ induces isomorphisms $P(\phi^{R'}) \isoto P(\phi)$, $L(\phi^{R'}) \isoto L(\phi)$ and $U(\phi^{R'}) \isoto U(\phi)$ of group schemes over $R''$.
  \end{enumerate}
\end{proposition}
\begin{proof}
  (i) By Lemma \ref{PhiExt}, the functor $\phi$ extends to a cocontinuous tensor functor $\tilde\phi\colon \Rep G \to \Fil(R')$ which is unique up to unique isomorphism. By \cite[II.2.0.02]{SaavedraRivano}, the functor $\Rep G \to \Rep G_{R'}$ is canonically isomorphic to the functor $\Rep G \to (\Rep G)_{R'}$ from Construction \ref{BaseChange}. Hence by the universal property given by Proposition \ref{TBCUP}, there is a cocontinuous $R'$-linear tensor functor $\tilde\phi^{R'}$, which is unique up to unique isomorphism, which makes the following diagram commute:
  \begin{equation*}
    \xymatrix{
      \Rep G \ar[rd]^{\tilde\phi} \ar[d] & \\
      \Rep G_{R'} \ar[r]_{\tilde\phi^{R'}} & \Fil(R').
      }
    \end{equation*}

    By composing $\tilde\phi^{R'}$ with the inclusion $\Rep^\circ G_{R'} \into \Rep G_{R'}$ we obtain a functor $\phi^{R'}$ as desired, which is however a priori only right exact. However, the fact that every element of $\Rep^\circ G_{R'}$ is dualizable implies that $\phi^{R'}$ is exact.

    From the various universal properties which we have used to construct it, one checks that any such functor $\phi^{R'}$ must arise in this way and that hence it is unique up to unique isomorphism.

    (ii) In the construction of $\phi^{R'}$ given above, the various universal properties show that the groups of tensor automorphisms of the functors $\phi$, $\tilde\phi$, $\tilde\phi^{R'}$ and $\phi^{R'}$ are canonically equal.

    (iii) Let $R'''$ be an $R''$-algebra. Then by (i) the filtered fiber functors $(\phi_{R'''})^{R'}$ and $(\phi^{R'})_{R'''}$ on $\Rep^\circ G_{R'}$ over $R'''$ are isomorphic. Applying (ii) two these fiber functors for varying $R'''$ shows that the induced homomorphism $P(\phi^{R'})\to P(\phi)$ is an isomorphism. Carrying out the analogue of the above arguments for graded instead of filtered fiber functors shows that $L(\phi^{R'}) \to L(\phi)$ is an isomorphism. This implies that $U(\phi^{R'}) \to U(\phi)$ is an isomorphism.
\end{proof}
\subsection{Splittings}
We continue with the notation from the previous subsection. We fix a scheme $S$ over $R$ and a filtered fiber functor on $\phi$ on $\Rep^\circ G$ over $S$. We denote the underlying fiber functor $\forg \circ \phi\colon \Rep^\circ G \to \Mod(S)$ of $\phi$ by $\omega$.
\begin{definition}
  \begin{enumerate}[(i)]
  \item A \emph{splitting} of $\phi$ is a graded fiber functor $\gamma$ on $\Rep^\circ G$ over $S$ for which $\phi = \fil\circ\gamma$. 
  \item The filtered fiber functor $\phi$ is \emph{splittable} if there exists such a splitting.
  \item The functor $\USpl(\phi) \colon (\textbf{Sch}/S) \to (\text{Sets})$ sends a scheme $S'$ over $S$ to the set of splittings of $\phi_{S'}$ and acts on morphisms by pullbacks.
  \end{enumerate}
\end{definition}

The following generalizes \cite[Lemma 4.10]{FFF} and is proved in exactly the same way:
\begin{lemma} \label{SplittingCochar}
  Giving a splitting of $\phi$ is the same as giving a homomorphism $\chi\colon D^\Gamma_S \to P(\phi)$ over $S$ whose composition with $\gr\colon P(\phi) \to L(\phi)$ is the homomorphism $\chi(\gr\circ\phi) \colon D^\Gamma_S \to L(\phi)$.
\end{lemma}

\begin{definition}
  We call a cocharacter $\chi$ as in Lemma \ref{SplittingCochar} a \emph{cocharacter which splits} $\phi$.
\end{definition}

Lemma \ref{SplittingCochar} gives a natural left action of $U(\phi)$ on $\USpl(\phi)$: Any conjugate of a cocharacter which splits $\phi$ by a section of $U(\phi)$ again splits $\phi$.

\begin{lemma}[{\cite[IV.2.2.1]{SaavedraRivano}}]  \label{SplPseudotorsor}
  This action turns $\USpl(\phi)$ into a $U(\phi)$-pseudotorsor, in the sense that for each scheme $S'$ over $S$, the group $U(\phi)(S')$ acts simply transitively on $\USpl(\phi)(S')$. 
\end{lemma}

The following is our main theorem:
\begin{theorem} \label{MainThm}
  Let $G$ be a smooth affine group scheme over $R$ which has enough dualizable representations. Let $\phi$ be filtered fiber functor on $\Rep^\circ G$ over some $R$-scheme $S$. Then $\phi$ is splittable on every affine open subset of $S$.
\end{theorem}

The proof of this result will be given in several steps. Clearly the claim for affine $S$ implies the claim for general $S$. Hence we assume from now on that $S=\Spec(R')$ for an $R$-algebra $R'$.

\subsubsection*{The case where $R$ is a field}
First we consider the case where $R$ is a field:

In case $\Gamma=\BZ$, the following is given by \cite[Theorem 4.15]{FFF}. However, the proof given in \emph{loc.cit.} also works for arbitrary index groups $\Gamma$.

\begin{theorem} \label{MTCase1}
  Assume that $R$ is a field and that $\phi$ is a filtered fiber functor on $\Rep^\circ G$ over $S$. Then $\phi$ is splittable. %In case $G$ is smooth and $S$ is affine, the filtered fiber functor $\phi$ is splittable.
\end{theorem}

\subsubsection*{\'Etale-local splittability}
Next we show that in the situation of Theorem \ref{MainThm}, the filtered fiber functor $\phi$ is splittable \'etale-locally on $S$. For this we will need the following results:

\begin{lemma} \label{SumLift}
  Let $M$ be a finitely generated module over a local ring $R$ with residue field $k$. Let $U$ and $V$ be two local direct summands of $M$. If the special fibers $U_k$ and $V_k$ give a direct sum decomposition $M_k=U_k \oplus V_k$, then $M=U\oplus V$.
\end{lemma}
\begin{proof}
  Since $R$ is local, the submodules $U$ and $V$ are direct summands of $M$. So there exist projections $\pi\colon M \to U$ and $\pi'\colon M \to V$. We consider the addition map $f\colon U \oplus V \to M$ and the map $g=(\pi,\pi') \colon M \to U\oplus V$. By our assumption, the special fibers of both $f$ and $g$ are surjective. Hence by Nakayama's lemma, the two compositions $f\circ g$ and $g \circ f$ are surjective. By \cite[Tag 05G8]{stacks-project}, any surjective endomorphism of a finitely generated module is an isomorphism. So $f\circ g$ and $g\circ f$ are isomorphisms, which implies that $f$ and $g$ are isomorphisms.
\end{proof}

The following result, as well as its proof, is essentially a repetition of a part of the proof of \cite[Theorem 5.8]{FFF}.

 \begin{lemma} \label{ExtendSplitting}
   Let $\phi$ be a filtered fiber functor on $\Rep^\circ G$ over $S$. Let $s \in S$  be a point. If the fiber $\phi_s$ is splittable, then there exists an \'etale neighborhood $S' \to S$ of $s$ such that $\phi_{S'}$ is splittable.
 \end{lemma}
 \begin{proof} 
Let $\chi_s$ be a cocharacter of $P(\phi)_s$ which splits $\phi_s$. By Theorem \ref{FpqcIso}, the group schemes $G_S$ and $\UAut^\otimes(\omega)$ are isomorphic fpqc-locally on $S$. Hence since $G$ is smooth so is $\UAut^\otimes(\omega)$. By \cite[Th\'eor\`eme XI.5.8]{SGA3II} this implies that there exist an \'etale morphism $S'\to S$, a point $s'\in S'$ over $s$ with trivial residue field extension $k(s')/k(s)$ and a cocharacter $\chi$ of $\UAut^\otimes(\omega)$ over $S'$ whose fiber in $s'$ is equal to $\chi_s$. 

We show that $\phi_{\CO_{S',s'}}$ is splittable: Let $\omega'$ be the fiber functor $(\forg\circ \phi)_{\CO_{S',s'}}$ on $\Rep^\circ G$ over $\CO_{S',s'}$. Then for $X \in \Rep^\circ G$ the filtered fiber functor $\phi$ gives a filtration $(\omega'(X)^{\phi\leq\gamma})_{\gamma\in\Gamma}$ on $\omega'(X)$. We write the grading of $\omega'(X)$ defined by $\chi$ as $\omega'(X)=\oplus_{\gamma\in\Gamma}\gr^\gamma_\chi(\omega'(X))$. This grading induces an decreasing filtration on $\omega'(X)$ by the submodules $\omega'(X)^{\chi\geq\gamma} \defeq \oplus_{\gamma'\geq \gamma}\gr^{\gamma'}_\chi(\omega'(X))$ for $\gamma \in \Gamma$.

Now for $X \in \Rep^\circ G$ and $\gamma\in \Gamma$ we claim that the submodules $\omega'(X)^{\phi\leq\gamma}$ and $\omega'(X)^{\chi>\gamma}\defeq \oplus_{\gamma'> \gamma}\gr^{\gamma'}_\chi(\omega'(X))$ of $\omega'(X)$ give a direct sum decomposition
\begin{equation} \label{PfDS}
  \omega'(X)=\omega'(X)^{\phi\leq\gamma} \oplus \omega'(X)^{\chi>\gamma} .
\end{equation}
Since by construction the cocharacter $\chi$ splits $\phi$ in $s'$, this is the case in the fiber above $s'$. Since $X \in \Rep^\circ G$ is dualizable, so is $\phi'(X)$, and hence $\omega'(X)^{\phi\leq\gamma}$ locally is a direct summand of $\omega'(X)$. So is $\omega'(X)^{\chi>\gamma}$ by construction. Hence Lemma \ref{SumLift} show that the decomposition \eqref{PfDS} exists.

As in \cite[Lemma 4.3]{FFF}, the existence of these direct sum decompositions for all $X$ and $\gamma$ shows that the two filtrations $(\omega'(X)^{\phi\leq\gamma})_{\gamma\in \Gamma}$ and $(\omega'(X)^{\chi\geq \gamma})_{\gamma\in\Gamma}$ are opposite in the sense that $$\gr^\gamma\omega'(X)\defeq \omega'(X)^{\phi\leq\gamma} \cap \omega'(X)^{\chi\geq\gamma}$$ defines a grading of $\omega'(X)$ which splits both filtrations. Since this grading is functorial in $X$ and compatible with tensor products we have constructed a graded fiber functor $\gamma'$ on $\Rep G$ over $\CO_{S',s'}$ which splits $\phi_{\CO_{S',s'}}$.

Since $P(\phi)$ is of finite type over $S$ by Proposition \ref{PhiAutRepr}, the homomorphism $D^\Gamma_{\CO_{S',s'}} \to P(\phi)_{\CO_{S',s'}}$ corresponding to this splitting under Lemma \ref{SplittingCochar} extends to to a homomorphism $D^\Gamma_U \to P(\phi)_U$ for some neighborhoud $U$ of $s'$ in $S'$. It follows from Lemma \ref{SplittingCochar} that after potentially shrinking $U$ this extension splits $\phi$ over $U$. This gives the desired \'etale neighborhood of $s$.
\end{proof}

Using this we prove:

\begin{theorem} \label{ELSplittable}
  In the situation of Theorem \ref{MainThm}, the filtered fiber functor $\phi$ is splittable \'etale-locally on $S$.
\end{theorem}

 \begin{proof}
   Consider a closed point $s$ of $S=\Spec(R)$ and the associated ring homomorphism $R \to k(s)$. By applying Theorem \ref{MTCase1} to the filtered fiber functor $\phi^{k(s)}$ given by Proposition \ref{PhiBC}, we see that the fiber $\phi_s$ is splittable. Hence by Lemma \ref{ExtendSplitting} there exists an \'etale neighborhood of $s$ on which $\phi$ is splittable. By varying $s$ over $S$ we obtain the claim.
 \end{proof}

 \subsubsection*{Splittability}
 To obtain splittability over $S$ itself, consider in more detail the structure of the automorphism group $P(\phi)$:
 
\begin{lemma} \label{PhiExtExact}
  The extension $\tilde\phi\colon \Rep G \to \Fil(S)$ from Lemma \ref{PhiExt} is exact. 
\end{lemma}
\begin{proof}
  Since it suffices to prove the claim \'etale-locally on $S$, by Theorem \ref{ELSplittable} we may assume that $\phi$ is splittable. Then the claim follows from the fact that every graded fiber functor on $\Rep^\circ G$ over $S$ extends to an exact functor $\Rep G \to \Gr_S$.
\end{proof}
\begin{definition}
  Let $\phi$ be a filtered fiber functor on $\Rep^\circ G$ over $S$. For an element $\alpha \in \Gamma^{\leq 0}$, we let $U_{\leq\alpha}(\phi)$ be the subgroup functor of $P(\phi)$ consisting of those sections which act as the identity on the quotient $\omega(X)^{\phi\leq\gamma}/\omega(X)^{\phi\leq \gamma+\alpha}$ for all $X \in \Rep^\circ G$ and all $\gamma\in\Gamma$.

  We also let $U_{<\gamma}(\phi) \defeq \colim_{\gamma'<\gamma} U_{\leq\gamma}(\phi) \subset P(\phi)$.
\end{definition}

Note that in particular $U_{\leq 0}(\phi)=P(\phi)$ and $U_{<0}(\phi)=U(\phi)$.

\begin{lemma} \label{UABC}
  In the situation of Proposition \ref{PhiBC}, the isomorphism $P(\phi^{R'}) \to P(\phi)$ restricts to isomorphisms $U_{\leq\alpha}(\phi^{R'}) \to U_{\leq\alpha}(\phi)$ for all $\alpha \in \Gamma^{\leq 0}$.
\end{lemma}
\begin{proof}
  By Lemma \ref{PhiExtExact}, the extension of $\phi$ (resp. $\phi^{R'}$) to $\Rep G$ (resp. $\Rep G_{R'}$) is exact. Since every representation $X \in \Rep^\circ G$ embeds into $\omega(X) \otimes \rho_{G} \in \Rep G$ via the comodule morphism, this implies that a section of $P(\phi)$ (resp. $P(\phi^{R'})$) is in $U_{\leq \alpha}(\phi)$ (resp. $U_{\leq\alpha}(\phi^{R'})$) if and only if it acts as the identity on $\omega(\rho_{G,R'})^{\phi\leq\gamma}/\omega(\rho_{G,R'})^{\phi\leq\gamma+\alpha}$ for all $\gamma \in \Gamma$. This proves the claim.
\end{proof}

\begin{lemma} \label{UADiscr}
  For each $\alpha \in \Gamma^{\leq 0}$ there exists an element $\alpha'<\alpha$ of $\Gamma$ for which $U_{<\alpha}(\phi)=U_{\leq\alpha'}(\phi)$.
\end{lemma}
\begin{proof}
  As in the proof of Lemma \ref{UABC}, whether a section of $P(\phi)$ lies in $U_{\leq \alpha}(\phi)$ can be seen from its action on $\omega(\rho_{G,R'})$. Since $\rho_G$ is finitely generated as an $R$-algebra and can be written as a colimit of dualizable representations, there exists an object $X \in \Rep^\circ G$ together with a map $X \to \rho_G$ whose image contains a generating set of the $R$-algebra $\rho_G$. For such an $X$ it follows that a section of $P(\phi)$ is in $U_{\leq \alpha}(\phi)$ if and only if it acts as the identity on $\omega(X)^{\phi\leq\gamma}/\omega(X)^{\phi\leq\gamma+\alpha}$ for all $\gamma \in \Gamma$. Using this, the fact that the filtration $(\omega(X)^{\phi\leq\gamma})_{\gamma \in \Gamma}$ has only finitely many steps implies the claim.
\end{proof}
\begin{proposition} \label{UaRepr}
  The subgroup functors $U_{\leq\alpha}(\phi)$ and $U_{<\alpha}(\phi)$ of $P(\phi)$ are representable by closed subschemes which are smooth over $S$.
\end{proposition}
\begin{proof}
  By Lemma \ref{UADiscr} it suffices to prove the claim for the functors $U_{\leq\alpha}(\phi)$. Since it suffices to prove the claim \'etale-locally on $S$, by Theorem \ref{ELSplittable} we may assume that $\phi$ is splittable and that $S=\Spec(R')$ is affine. Then using Lemma \ref{UABC} we may replace $\phi$ by $\phi^{R'}$ and hence assume that $S=\Spec(R)$.

  By Theorem \ref{FpqcIso}, the group schemes $G$ and $\UAut^\otimes(\omega)$ differ by a $G$-torsor on $S$. Using the induced isomorphism $\Rep^\circ G \cong \Rep^\circ \UAut^\otimes(\omega)$ we may replace $G$ by $\UAut^\otimes(\omega)$ and hence can assume that $\omega=\forg\colon \Rep^\circ G \to \Mod(R)$. Then the claim is given by \cite[IV 2.1.4.1]{SaavedraRivano}.
\end{proof}

\begin{construction}
 Assume that $S=\Spec(R)$. Let $\alpha \in \Gamma^{<0}$. Let $g \in U_{\leq\alpha}(\phi)(S')$ for some $S$-scheme $S'$. For $X \in \Rep^\circ G$, the endomorphism $h_X \defeq g - \id_{\omega(X)_{S'}}$ of $\omega(X)_{S'}$ maps $\omega(X)^{\phi\leq\gamma}_{S'}$ into $\omega(X)^{\phi\leq\gamma+\alpha}_{S'}$ for each $\gamma \in \Gamma$. Hence it naturally induces a homomorphism $\tilde h_X \colon \gr(\phi(X)) \to \gr(\phi(X))$ of degree $\alpha$. For $X, Y \in \Rep^\circ G$, the identity
  \begin{equation*}
    h_{X\otimes Y}= h_X \otimes \id_{\omega(Y)_{S'}} + \id_{\omega(X)_{S'}}\otimes h_{Y} + h_X \otimes h_Y
  \end{equation*}
  implies that the $\tilde h_X$ are compatible with tensor products in the sense of Definition \ref{TEnd}. Hence for varying $X$ these endomorphisms define an element $$\ell(g) \in \UEnd^\otimes(\forg\circ\gr\circ\phi)(S').$$

  By Theorem \ref{FpqcIso}, the fiber functor $\forg\circ\gr\circ\phi$ on $\Rep^\circ G$ is isomorphic to $\omega_G$ fpqc-locally on $X$. Hence $\UAut^\otimes(\forg\circ\gr\circ\phi)$ is representable by a smooth group scheme over $S'$. So using Proposition \ref{TannakaLie} we can consider $\ell(g)$ as an element
  \begin{equation*}
    \ell(g) \in \ULie(\UAut^\otimes(\forg\circ\gr\circ\phi))(S')=\Lie(\UAut^\otimes(\forg\circ\gr\circ\phi))_{S'}.
  \end{equation*}

  Using the equivariance statement in Proposition \ref{TannakaLie}, the fact that the $\tilde h_X$ are of degree $\alpha$ translates to the fact that $\ell(g)$ lies in the subspace $\gr_\alpha(\Lie(\UAut^\otimes(\forg\circ\gr\circ\phi))_{S'}$ of elements of degree $\alpha$ with respect to the grading of $\Lie(\UAut^\otimes(\forg\circ\gr\circ\phi))$ defined by $\gr\circ\phi$ via the adjoint action. The assignment $g \mapsto \ell(g)$ is functorial in $S'$ and so we obtain a morphism
  \begin{equation} \label{LogMorphism}
    \ell \colon U_{\leq \alpha}(\phi) \to \mathbb{V}(\gr_\alpha \Lie(\UAut^\otimes(\forg\circ\gr\circ\phi)))
  \end{equation}
  of schemes over $S$.

\end{construction}
\begin{lemma} \label{LogKernel}
  Let $\alpha \in \Gamma^{<0}$.
  \begin{enumerate}[(i)]
  \item The morphism \eqref{LogMorphism} is a group homomorphism with respect to the additive group structure on its target.
  \item The kernel of \eqref{LogMorphism} is equal to $U_{<\alpha}(\phi)$.
  \end{enumerate}
\end{lemma}
\begin{proof}
  (i) For sections $g,h \in U_{\leq\alpha}(\phi)(S')$ and for an element $x \in F^{\gamma}(\omega(X)_{S'})$ we find
  \begin{equation*}
    (g(h(x))-x)-(g(x)-x)-(h(x)-x)=g(h(x)-x)-(h(x)-x) \in \omega(X)^{\phi\leq \gamma+2\alpha}_{S'}.
  \end{equation*}
  This implies (i).

  Claim (ii) follows from the definition of \eqref{LogMorphism}.  
\end{proof}

% The proof of the following fact is motivated by the proof of \cite[2.3.1]{ConradReductiveGroupSchemes}:
% \begin{proposition} \label{FibrewiseIso}
%   Consider a homomorphism $h\colon G \to H$ of smooth affine group schemes for which every geometric fiber of the induced inclusion $i\colon G/\ker(h) \to H$ is an isomorphism. Then $i$ is an isomorphism.
% \end{proposition}
% \begin{proof}
%   The quotient sheaf $G/\ker(h)$ is representable by a separated algebraic space of finite type over $S$ (e.g. since it is equal to the quotient stack $[G / \ker(h)]$). The assumption on the fibers of $i$ implies that $i$ is quasi-finite. Since $i$ is a monomorphism and $G/\ker(h)$ is separated, the morphism $i$ is separated as well. Hence by \cite[Tag 03XX]{stacks-project}, the algebraic space $G/\ker(h)$ is representable by a scheme. Since $G$ is smooth over $S$, so is $G/\ker(h)$. By the assumption on $i$, the induced map $\Lie(G/\ker(h)) \to \Lie(H)$ is a isomorphism on every geometric fiber. Hence it is an isomorphism, which shows that $i$ is \'etale. Hence, by \cite[Tag 025G]{stacks-project}, as an \'etale monomorphism $i$ is an open immersion. Then the assumption on the fibers of $i$ ensures that $i$ is an isomorphism.
% \end{proof}

\begin{theorem} \label{GrADesc}
  For each $\alpha < 0$, the quotient sheaf $U_{\leq\alpha}(\phi)/U_{<\alpha}(\phi)$ is representable by the affine group scheme underlying a vector bundle over $S$.

  More precisely, the morphism \eqref{LogMorphism} induces an isomorphism
    \begin{equation*} 
    U_{\leq \alpha}(\phi)/U_{<\alpha}(\phi) \isoto \mathbb{V}(\gr_\alpha \Lie(\UAut^\otimes(\forg\circ\gr\circ\phi))).
  \end{equation*}

\end{theorem}
The following proof was simplified by a suggestion of an anynomous referee.

\begin{proof}
  Since by Lemma \ref{LogKernel}, the morphism \ref{LogMorphism} has kernel $U_{<\alpha}(\phi)$, the claim is equivalent to the statement that the morphism \eqref{LogMorphism} is faithfully flat. For this we may work \'etale locally on $S$ and so by Theorem \ref{ELSplittable} we may assume that $S$ is splittable. By Proposition \ref{UaRepr}, the group scheme $U_{\leq\alpha}(\phi)$ is smooth. Hence by the fiberwise criterion for flatness, it suffices to verify the faithful flatness of \eqref{LogMorphism} on each geometric fiber over $\Spec(R)$. So using Lemma \ref{UABC} we may reduce to the case that $R$ is a field.

  Then \eqref{LogMorphism}  factors through a monomorphism
      \begin{equation*} 
    U_{\leq \alpha}(\phi)/U_{<\alpha}(\phi) \into \mathbb{V}(\gr_\alpha \Lie(\UAut^\otimes(\forg\circ\gr\circ\phi)))
  \end{equation*}
  which is representable by a closed immersion, and it suffices to show that the source and target of this closed immersion have the same dimension. This is given by  \cite[IV 2.1.4.1]{SaavedraRivano}.
\end{proof}

Now we can prove Theorem \ref{MainThm}:
\begin{proof}
  By Lemma \ref{SplPseudotorsor} and Theorem \ref{ELSplittable}, the sheaf $\USpl(\phi)$ is a torsor under $U(\phi)$. By Theorem \ref{GrADesc}, the group scheme $U(\phi)$ is an iterated extension of vector group schemes over $S$. Hence any $U(\phi)$-torsor over an affine base is trivial.
\end{proof}

%%%%%%%%%%%%%%%%%%%%%%
\bibliography{references}
\bibliographystyle{alpha}

\end{document}